\documentclass[11pt]{amsart}
\usepackage[colorlinks=true,pagebackref,hyperindex,citecolor=green,linkcolor=red]{hyperref}
\usepackage{amsmath}
\usepackage{amsfonts}
\usepackage{amssymb}
\usepackage{color}
\usepackage[all]{xy}  
\usepackage{enumerate}
\usepackage[top=1in, bottom=1in, left=1in, right=1in]{geometry}
\usepackage{mathrsfs}

\usepackage{stmaryrd}

\theoremstyle{definition}
\newtheorem{theorem}{Theorem}[section]

\newtheorem{questions}[theorem]{Questions}
\newtheorem{corollary}[theorem]{Corollary}
\newtheorem{lemma}[theorem]{Lemma}
\newtheorem{proposition}[theorem]{Proposition}

\theoremstyle{definition}
\newtheorem{definition}[theorem]{Definition}

\newtheorem{example}[theorem]{Example}

\newtheorem*{conjecture}{Conjecture}
\newtheorem{remark}[theorem]{Remark}
\numberwithin{equation}{subsection}

\newcommand{\m}{\mathfrak{m}}

\newcommand{\NN}{\mathbb{N}}
\newcommand{\ZZ}{\mathbb{Z}}
\newcommand{\QQ}{\mathbb{Q}}
\newcommand{\FF}{\mathbb{F}}

\newcommand{\MM}{\mathcal{M}}

\newcommand{\Spec}{\operatorname{Spec}}
\newcommand{\Depth}{\operatorname{depth}}
\newcommand{\Hom}{\operatorname{Hom}}

\newcommand{\Ext}{\operatorname{Ext}}

\newcommand{\Min}{\operatorname{min}}

\newcommand{\ls}{\leqslant}
\newcommand{\gs}{\geqslant}
\newcommand{\ds}{\displaystyle}

\newcommand{\p}{\mathfrak{p}}

\newcommand{\ov}[1]{\overline{#1}}

\newcommand{\n}{\mathfrak{n}}

\newcommand{\ps}[1]{\llbracket {#1} \rrbracket}
\begin{document}
\newcommand{\tens}{\otimes}
\newcommand{\hhtest}[1]{\tau ( #1 )}
\renewcommand{\hom}[3]{\operatorname{Hom}_{#1} ( #2, #3 )}
\newcommand{\ind}{\operatorname{index}}
\newcommand{\gll}{\operatorname{g\ell\ell}}
\newcommand{\ord}{\operatorname{ord}}
\renewcommand{\ll}{\operatorname{\ell\ell}}
\newcommand{\soc}{\operatorname{soc}}
\newcommand{\CCC}{\mathfrak{C}}
\newcommand{\frk}{\operatorname{f-rank}}

\title{A counterexample to a Conjecture of Ding}
\author{Alessandro De Stefani}


\maketitle
\begin{abstract} We give a counterexample to a conjecture posed by S. Ding in \cite{Ding} regarding the index of a Gorenstein local ring by exhibiting several examples of one dimensional local complete intersections of embedding dimension three with index $5$ and generalized L{\"o}ewy length $6$.
\end{abstract}
\section{Introduction}
Let $(R,\m,k)$ be a Cohen-Macaulay local ring of Krull dimension $d$. For a finitely generated $R$-module $M$, Auslander's $\delta$ invariant $\delta(M)$ of $M$ is defined as the smallest integer $m$ such that there is an epimorphism $X \oplus R^m \to M \to 0$ with $X$ a maximal Cohen-Macaulay module with no free direct summands. In the case $M=R/\m^n$, where $n \gs 1$ is an integer, it can be shown that $\delta(R/\m^n) \ls 1$ for all $n$, and that if equality holds for some $n$, then it holds for all $i \gs n$. The index of $R$ is defined as $\ind(R): =\inf\{n \mid \delta(R/\m^n) = 1\}$. For a module $M$ of finite length let $\ll(M)$ be its L{\"o}ewy length, that is the smallest integer $n$ such that $\m^nM = 0$. For a finitely generated $R$-module $M$ of dimension $c$ define the generalized L{\"o}ewy length as $\gll(M):= \min\{\ll(M/(x_1,\ldots,x_c)M) \mid x_1,\ldots,x_c$  is a system of parameters for  $M\}$. In \cite{Ding} Ding posed the following conjecture.
\begin{conjecture} Let $(R,\m,k)$ be a Gorenstein local ring of dimension $d$. Then $\ind(R) = \gll(R)$.
\end{conjecture}
Herzog showed that the conjecture is true for homogeneous Gorenstein algebras over an infinite field \cite{HerzogIndex} (extending the concepts in an obvious way to the graded case). Later, Ding generalized this result proving that it holds true if the associated graded ring ${\rm gr}_\m(R)$ is Cohen-Macaulay \cite[Theorem 2.1]{Ding2}. Hashimoto and Shida pointed out in \cite{HaSh} that Ding's result needs the residue field $k$ to be infinite, and the conjecture may fail for rings with finite residue field, even if ${\rm gr}_\m(R)$ is Cohen-Macaulay \cite[Example 3.2]{HaSh}. In fact, what is really needed in order to apply Ding's Theorem is that there exists a homogeneous maximal regular sequence $x_1^*, \ldots, x_d^*$ in $({\rm gr}_\m(R))_1$. In Section 2 we recover this result in the one dimensional case, using a different approach (Corollary \ref{CM}). In \cite[Theorem 3.15]{YoshinoAusl}, Yoshino claims that Ding proves that the conjecture is true for all Gorenstein local rings such that $\Depth({\rm gr}_\m(R)) \gs \dim(R)-1$. We could not find a reference for this claim, and the examples in this paper, being all one dimensional, show that this statement is indeed false. See Remark \ref{remDing} for additional comments. In Section 3 we show that if $k$ is any field, then the one dimensional complete intersection domain
\[
%\ds R= \left(\frac{k[x,y,z]}{(x^2-y^5,xy^2-yz^3-z^5)}\right)_{(x,y,z)}
\ds R= \left(\frac{k[x,y,z]}{(x^2-y^5,xy^2+yz^3-z^5)}\right)_{(x,y,z)}
\]
has index 5 and generalized L{\"o}ewy length 6. This provides a counterexample to Ding's conjecture, without any assumptions on the residue field $k$. 
%When $k = \QQ$, we check that $R$ is a domain using Macaulay2.
In Section 4 we give a counterexample in which $R$ quasi-homogeneous (Example \ref{ex1}), and a counterexample where $R$ is analytically irreducible (Example \ref{ex2}). All the computations in this article are made with CoCoA \cite{CoCoa} and Macaulay2 \cite{Mac2}.
\section{Preliminaries and results for one dimensional rings}
Throughout, $(R,\m,k)$ will denote a commutative Noetherian local ring with identity and $M$ will denote a finitely generated $R$-module. We first recall some known facts about superficial elements and Hilbert functions. We include some of the proofs for convenience of the reader. In the following, $\lambda(M)$ denotes the length of an $R$-module $M$, and $e(R)$ denotes the multiplicity of $R$ with respect to the maximal ideal $\m$. Also, for a non-zero element $x \in R$, $\ord(x)$ denotes the order of $x$, that is the largest integer $n$ such that $x \in \m^n$. 
\begin{definition} Let $(R,\m,k)$ be a local ring and let $x \in \m$ with $\ord(x) = d \gs 1$. Then $x$ is said to be a {\it superficial element} (of order $d$) if there exists an integer $c$ such that $(\m^{n+d}:x) \cap \m^c = \m^n$ for all $n \gs c$.
\end{definition}
\begin{lemma} \label{sup} Let $(R,\m,k)$ be a one dimensional Cohen-Macaulay local ring, and let $x \in \m$ be a non zero-divisor. If $d = \ord(x)$, then
\[
\ds \lambda(R/(x)) \gs d \cdot e(R).
\]
Furthermore, equality holds if and only if $x$ is a superficial element.
\end{lemma}
\begin{proof}
For all $n \gs 1$ we have an exact sequence
\[
\xymatrixrowsep{3mm}
	\xymatrixcolsep{8mm}
\xymatrix{ 
\ds 0 \ar[r] &\ds \frac{\m^{n+d}:x}{\m^n} \ar[r] & \ds \frac{R}{\m^n} \ar[r]^-{\cdot x} & \ds \frac{R}{\m^{n+d}} \ar[r] & \ds\frac{R}{\m^{n+d} + (x)} \ar[r] & 0.
}
\]
As $\dim(R) = 1$, there exists an integer $N$ such that $\lambda\left(\m^n/\m^{n+1}\right) = e(R)$ and $\m^{n+d} \subseteq (x)$ for all $n \gs N$. For $n \gs N$ we have
\[
\ds \lambda(R/(x)) = \lambda\left(\frac{R}{\m^{n+d}+(x)}\right) = \lambda\left(\frac{\m^n}{\m^{n+d}}\right) + \lambda\left(\frac{\m^{n+d}:x}{\m^n}\right) = d \cdot e(R) + \lambda\left(\frac{\m^{n+d}:x}{\m^n}\right) \gs d \cdot e(R).
\]
Finally, equality holds if and only if $\m^{n+d}:x = \m^n$ for all $n \gg 0$, and by \cite[Lemma 8.5.3]{HunSw} this happens if and only if $x$ is a superficial element.
\end{proof}
\begin{remark} \label{par} In the above setting, a non zero-divisor $x \in R$ with $\ord(x) = d$ is a superficial element if and only if $x$ is a reduction of $\m^d$, if and only if the initial form $x^*$ is a homogeneous parameter of degree $d$ in the associated graded ring ${\rm gr}_m(R)$. See \cite{RossiValla} or \cite{HunSw} for more details.
\end{remark}
Given a local ring $(R,\m,k)$, the {\it Hilbert function of $R$}, denoted $HF_R$, is defined to be the Hilbert function of the associated graded ring ${\rm gr}_\m(R) = \bigoplus_{n\gs 0} \ \m^n/\m^{n+1}$, that is
\[
\ds HF_R(n):= \lambda(\m^n/\m^{n+1}) \ \ \ \ \mbox{ for all } n \gs 0.
\]
For given integers $d,n$ there is a unique expression $d= {k_n \choose n} + {k_{n-1} \choose n-1} + \ldots + {k_1 \choose 1}$ with $k_n > k_{n-1} > \ldots > k_1 \gs 0$. Define $d^{\langle n \rangle} := {k_n+1 \choose n+1} + {k_{n-1}+1 \choose n} + \ldots + {k_1+1 \choose 2}$. The following is a well known Theorem of Macaulay on Hilbert functions of standard graded algebras over a field.
\begin{theorem}(\cite{Mac}, \cite[Theorem 4.2.10]{BrHe}) \label{Mac} Let $k$ be a field and let $A$ be a standard graded $k$-algebra. Then 
\[
\ds HF_A(n+1) \ls HF_A(n)^{\langle n \rangle} \ \ \ \mbox{ for all } n \gs 1.
\]
\end{theorem}
\begin{remark} Let $P=k[x_1,\ldots,x_n]$ be a polynomial ring over a field, and let $\n=(x_1,\ldots,x_n)$. Assume that $Q$ is either the completion $\widehat P$ with respect to $\n$, or the localization $P_\n$. Let $R:=Q/I$ for some ideal $I \subseteq Q$, and let $\m$ be the maximal ideal of $R$. Then it is well know that ${\rm gr}_\m(R) \cong P/I^*$, where $I^* = (f^* \mid f \in I)$ is the {\it initial ideal of $I$}, that is the homogeneous ideal of $P$ generated by the initial forms $f^*$ of elements $f \in I$.
\end{remark}

We recall here some known results about the delta invariant and the index. For more details and complete proofs about these facts we refer the reader to \cite{Ding} and \cite[Chapter 11]{LW}.
\begin{definition} Let $(R,\m)$ be a Cohen-Macaulay local ring and let $M$ be a finitely generated $R$-module. The {\it Auslander delta invariant} of $M$ is defined as
\[
\ds \delta(M) = \min\{\frk(X) \mid X \mbox{ is MCM  and there exists a surjection } X \to M \to 0\},
\]
where $\frk(X)$ denotes the free rank of $X$, i.e. the maximal number of copies of $R$ splitting out of $X$, and MCM stands for Maximal Cohen-Macaulay (i.e. $\Depth(X) = \dim(R)$).
\end{definition}
Using the notion of minimal MCM approximation, originally due to Auslander and Buchweitz \cite{AusBuch}, one can compute $\delta(M)$ by looking at a specific choice of a MCM module mapping onto $M$. A MCM approximation of a finitely generated $R$-module $M$ is an exact sequence
\[
\xymatrixrowsep{3mm}
	\xymatrixcolsep{8mm}
\xymatrix{ 
\ds 0 \ar[r] & Y_M \ar[r]^-i & X_M \ar[r] & M \ar[r] & 0,
}
\] 
where $X_M$ is a finitely generated MCM $R$-module and $Y_M$ is a finitely generated $R$-module of finite injective dimension. The approximation is called minimal if $Y_M$ and $X_M$ have no common direct summand via $i$ (see \cite[Lemma 11.12 and Proposition 11.14]{LW} for other equivalent definitions of minimality). When $R$ is Cohen Macaulay, minimal MCM approximations of a finitely generated  $R$-module $M$ always exist, and they are unique up to isomorphism of short exact sequences inducing the identity on $M$ \cite[Theorem 11.17 and Proposition 11.13]{LW}.
\begin{proposition}[\cite{LW}, Proposition 11.27] If $0 \to Y_M \to X_M \to M \to 0$ is a minimal MCM approximation of $M$ then
\[
\ds \delta(M) = \frk(X_M).
\] 
\end{proposition}
\begin{remark} \label{props delta} It is easy to see from the definition that if $M$ and $N$ are two $R$-modules and there is a surjection $N \to M \to 0$, then $\delta(M) \ls \delta(N)$. In particular, this implies that $\delta(M) \ls \mu(M)$, where $\mu(M)$ denotes the minimal number of generators of $M$.
\end{remark}
We now focus on $\delta(M)$ for some special choices of $M$: for any integer $n \gs 1$ consider $\delta(R/\m^n)$. It follows from Remark \ref{props delta} that
\[
0 \ls \delta(R/\m) \ls \delta(R/\m^2) \ls \ldots \ldots \ls \delta(R/\m^n) \ls \delta(R/\m^{n+1}) \ls \ldots\ldots \ls 1,
\]
therefore, it makes sense to define
\[
\ds \ind(R) := \inf\{n \mid \delta(R/\m^n) =1\}.
\]
Assuming that $R$ has a canonical module, Ding gives a complete characterization of rings for which the index is finite.
\begin{theorem}[Theorem 1.1, \cite{Ding}] Let $(R,\m)$ be a Cohen Macaulay local ring which has a canonical module. Then $\ind(R)<\infty$ if and only if $R_\p$ is Gorenstein for all $\p \in \Spec(R)\smallsetminus\{\m\}$.
\end{theorem}
If $R$ is Gorenstein and $x_1,\ldots,x_d$ is a maximal regular sequence in $R$, the beginning of a minimal free resolution of $\ov R:=R/(x_1,\ldots,x_d)$ is
\[
\xymatrixrowsep{3mm}
	\xymatrixcolsep{8mm}
\xymatrix{ 
\ds 0 \ar[r] & \Omega_1 \ar[r] & R \ar[r] & \ov R \ar[r] & 0,
}
\] 
and $\Omega_1$ has finite injective dimension since $R$ and $\ov R$ do. Therefore the one above is a MCM approximation of $\ov R$, and it is easily seen to be minimal. Therefore $\delta(R/(x_1,\ldots,x_d)) = 1$ for any choice of a maximal regular sequence $x_1,\ldots,x_d$ in $R$. In particular, if for some $n$ one has $\m^n \subseteq (x_1, \ldots, x_d)$, then there is a surjection $R/\m^n \to \ov R \to 0$, and this gives $\delta(R/\m^n) = 1$ as well. This shows that, when $R$ is Gorenstein, $\ind(R) \ls \ll(R/(x_1,\ldots,x_d))$ for any $x_1,\ldots,x_d$ system of parameters in $R$, and thus $\ind(R) \ls \gll(R)$. In \cite{Ding}, Ding conjectured that equality holds.

In the following we focus our attention on one dimensional rings, where it is easier to describe the index of $R$ because of the following proposition.
\begin{proposition} \label{colon} Let $(R,\m,k)$ be a one dimensional Gorenstein local ring and let $x \in \m$ be a non zero-divisor. For $n \gs 1$, the following facts are equivalent
\begin{enumerate}[(i)]
\item \label{1} $\delta(R/\m^n) = 1$.
\item \label{2} $x^n \in \m((x^n):\m^n)$.
\item \label{3} $x\m^n : \m \subseteq (x)$.
\item \label{4} $\Delta \notin x\m^n:\m$ for any $\Delta \in (x):\m$, $\Delta \notin (x)$.
\end{enumerate}
\end{proposition}
\begin{proof} It is shown in \cite[Proposition 1.2]{Ding1} that in our assumptions
\[
\delta(R/\m^n) = 1 + \mu(\Ext^1_R(R/\m^n,R)) - \mu(\Hom_R(\m^n,R)).
\]
Therefore $\delta(R/\m^n) = 1$ if and only if $\mu(\Ext^1_R(R/\m^n,R)) = \mu(\Hom_R(\m^n,R))$. Notice that 
\[
\ds \mu(\Ext^1_R(R/\m^n,R)) = \mu(\Hom_R(R/\m^n,R/(x^n))) = \mu\left(\frac{(x^n):\m^n}{(x^n)}\right).
\]
On the other hand, $\Hom_R(\m^n,R) \cong (x^n):\m^n$, and putting these facts together we have that $\delta(R/\m^n) = 1$ if and only if $\mu\left(\frac{(x^n):\m^n}{(x^n)}\right) = \mu((x^n):\m^n)$, if and only if $x^n \in \m((x^n):\m^n)$. This shows that (\ref{1}) and (\ref{2}) are equivalent. Using that $R$ is Gorenstein, by duality (\ref{2}) holds if and only if
\[
\ds (x) = (x^{n+1}) : (x^n) \supseteq (x^{n+1}) : \left(\m((x^n):\m^n)\right) = \left((x^{n+1}) : \left((x^n):\m^n\right)\right):\m = x\m^n : \m,
\]
therefore (\ref{2}) is equivalent to (\ref{3}). Finally, (\ref{3}) clearly implies (\ref{4}). For the converse, assume that (\ref{3}) does not hold, that is $x\m^n : \m \not\subseteq (x)$. Consider the following short exact sequence
\[
\xymatrixrowsep{3mm}
	\xymatrixcolsep{8mm}
\xymatrix{ 
\ds 0 \ar[r] & (x)/x\m^n \ar[r] & R/x\m^n \ar[r] & R/(x) \ar[r] & 0,
}
\] 
which induces an exact sequence on socles
\[
\xymatrixrowsep{3mm}
	\xymatrixcolsep{8mm}
\xymatrix{ 
\ds 0 \ar[r] & \ds \frac{(x)\cap (x\m^n:\m)}{x\m^n} \ar[r]^-\psi &\ds \frac{x\m^n: \m}{x\m^n} \ar[r]^-\varphi & \ds \frac{(x):\m}{(x)}.
}
\] 
Note that the last module, that is $\soc(R/(x))$, is isomorphic to $k$, therefore $\varphi$ is either surjective or it is zero. Since we are assuming that $x\m^n:\m \not\subseteq (x)$, we have that $(x) \cap (x\m^n:\m) \subsetneq x\m^n:\m$. Therefore $\psi$ is not an isomorphism, and thus $\varphi$ is surjective. This means that there exists a choice of $\Delta \in (x):\m$, $\Delta \notin (x)$, such that $\Delta \in x\m^n:\m$, proving that (\ref{4}) does not hold.
\end{proof}
As a corollary, we easily recover Ding's conjecture for one dimensional Gorenstein rings with infinite residue field, and whose associated graded ring is Cohen-Macaulay \cite[Theorem 2.1]{Ding2}.
\begin{corollary} \label{CM} Let $(R,\m,k)$ be a one dimensional Gorenstein local ring, and assume that the associated graded ring ${\rm gr}_\m(R)$ has a homogeneous non zero-divisor of degree one. This condition is satisfied, for example, if ${\rm gr}_\m(R)$ is Cohen-Macaulay and $k$ is infinite. Then $\ind(R) = \gll(R)$.
\end{corollary}
\begin{proof} We only have to show that $\ind(R) \gs \gll(R)$ as the other inequality always holds. Let $n=\ind(R)$ and let $x\in \m$ be a non zero-divisor such that $x^*$ is a non zero-divisor in ${\rm gr}_\m(R)$ of degree one. In particular, $\ord(x) = 1$. By Proposition \ref{colon} (\ref{3}) we have that $x\m^n:\m \subseteq (x)$. By way of contradiction suppose that $\m^n \not\subseteq (x)$, so that we can choose $\Delta \in \m^n$ which represents a non-zero socle element in $R/(x)$. Then $\Delta \m \subseteq \m^{n+1} \cap (x) = x\m^{n}$ by Valabrega-Valla's Theorem \cite[Theorem 2.3]{VV}, because $x^*$ is a non zero-divisor in $({\rm gr}_\m(R))_1$. Hence $\Delta \in x\m^n : \m \subseteq (x)$, and this contradicts the fact that $\Delta$ is chosen to be non-zero in $R/(x)$. Thus $\m^n \subseteq (x)$ and $\gll(R) \ls n = \ind(R)$.
\end{proof}
\begin{remark} We point out that in Example 3.2 of \cite{HaSh}, which is $R=\FF_2\ps{x,y}/(xy(x+y))$, there does not exist a homogeneous linear non zero-divisor in ${\rm gr}_{(x,y)R}(R)$, and thus Corollary \ref{CM} (or \cite[Theorem 2.1]{Ding2}) cannot be applied.
\end{remark}

\section{The counterexample} 
If $(R,\m,k)$ is a one dimensional Gorenstein local ring, the index of $R$ can be checked on any non zero-divisor $x \in R$, just by testing whether any of the equivalent conditions in Proposition \ref{colon} is satisfied. On the other hand, a generic choice of a non zero-divisor $x \in R$ will give a maximal value of $\ll(R/(x))$, whereas $\gll(R)$ is defined as the minimum of such L{\"o}ewy lengths. Therefore, to verify that a potential candidate is a counterexample, one should take in account the L{\"o}ewy length of $R/(x)$ for every non zero-divisor $x \in R$. Here is the counterexample.

Let $S=k[x,y,z]_{(x,y,z)}$, where $k$ is a field, and let $\n=(x,y,z)S$ be its maximal ideal. Consider
\[
\ds I=(x^2-y^5,xy^2+yz^3-z^5)S
\]
and let $R:=S/I$. We now want to show that $R$ is a domain. We refer to \cite[Section 19.5]{Conv_Disc_Geom} for more details about some notions and results that we are about to use.

Let $k$ be a field. For a polynomial $f \in k[x,y]$, the Newton polytope $N_f$ of $f$ is the convex hull of all points $(i,j) \in \ZZ^2$, where $ax^iy^j$ appears as a monomial in $f$ with non-zero coefficient $a \in k$. We say that $N_f$ is integer reducible if it can be written as the sum of two convex lattice polygons, each consisting of more than one point. This  means that we can write $N_f=A+B = \{(a_1+b_1,a_2+b_2) \mid (a_1,a_2) \in A, (b_1,b_2)\in B\}$, where $A$ and $B$ are convex lattice polygons containing at least two integer points. Otherwise, $N_f$ is integer irreducible. If the polynomial $f$ is not divisible by either $x$ or $y$, and $N_f$ is integer irreducible, then $f$ is irreducible \cite[Theorem 19.7 and the following Remark]{Conv_Disc_Geom}. Furthermore, if a convex lattice polygon $N$ has en edge of the form $[(0,m),(n,0)]$, with $m$ and $n$ relatively prime, and $N$ is contained in the triangle with vertices $(0,m),(n,0),(0,0)$, then $N$ is integer irreducible \cite[Corollary 19.2]{Conv_Disc_Geom}. 
\begin{lemma} \label{prime} For any field $k$, the ideal $J:=(x^2-y^5,xy^2+yz^3-z^5) \subseteq k[x,y,z]=:T$ is prime.
\end{lemma}
\begin{proof} Since $y \in T$ is a non-zero divisor modulo $J$, it suffices to show that $JT[y^{-1}]$ is a prime in the localization $T[y^{-1}]$ of $T$ at the element $y$. After some algebraic manipulations, we obtain
\[
\ds JT[y^{-1}] = (x^2-y^5,x+y^{-1}z^3-y^{-2}z^5)T[y^{-1}] = ((-y^{-1}z^3+y^{-2}z^5)^2-y^5,x+y^{-1}z^3-y^{-2}z^5)T[y^{-1}].
\]
Set $x':=x+y^{-1}z^3-y^{-2}z^5$, then we have
\[
\ds JT[y^{-1}] = (y^{-2}z^6-2y^{-3}z^8+y^{-4}z^{10}-y^5,x') T[y^{-1}] = (z^{10}-2z^8y+z^6y^2-y^9,x')T[y^{-1}].
\]
Note that $T[y^{-1}] = k[x,y,z,y^{-1}] = k[x',y,z,y^{-1}]$. Since $x' \in JT[y^{-1}]$, we have that $JT[y^{-1}]$ is prime if and only if $(z^{10}-2z^8y+z^6y^2-y^9)Q[y^{-1}]$ is prime, where $Q=k[y,z]$. Because $y \in Q$ is a non zero-divisor modulo $(z^{10}-2z^8y+z^6y^2-y^9)Q$, we finally reduce the claim to showing that the polynomial $f= z^{10}-2z^8y+z^6y^2-y^9$ is irreducible in $Q$. The exponents of the pure powers of $z$ and $y$ are $10$ and $9$, so they are relatively prime. Furthermore, the pairs of exponents $(8,1)$ and $(6,2)$ of the other monomials $-2z^8y$ and $z^6y^2$ appearing in $f$ are contained in the triangle with vertices $(0,0)$, $(0,9)$, $(10,0)$. It follows from the discussion preceding the lemma that $f$ is irreducible.
\end{proof}
Lemma \ref{prime} shows that, for any field $k$, the ring $R=S/I$ defined above is a domain. Denote by $\m$ the maximal ideal $\n/I$ of $R$. The ring $R$ is a one dimensional complete intersection. Its associated graded ring with respect to $\m$ is $G:={\rm gr}_\m(R) \cong P/I^* $, where $P=k[X,Y,Z]$ and
\[
\ds I^* = (X^2,XY^2,XYZ^3,YZ^6) = (X^2,Y) \cap (X,Z^6) \cap (X^2,Y^2,Z^3) \subseteq P.
\]
From now on, we will identify $G$ with $P/I^*$. The Hilbert function of $R$ is
\begin{eqnarray*}
\begin{tabular}{c|c|c|c|c|c|c|c|c|c|c}
\ $n$ \ & \ $0$ \ & \ $1$ \ & \ $2$ \ & \ $3$ \ & \ $4$ \ & \ $5$ \ & \ $6$ \ & \ $7$ \ & \ $8$ \ & $\ldots$ \\
\hline
$HF_R(n)$ & $1$ & $3$ & $5$ & $6$ & $7$ & $7$ & $8$ & $8$ & $8$ & $\ldots$
\end{tabular}
\end{eqnarray*}
with $HF_R(n) = 8 = e(R)$ for all $n \gs 6$.
\begin{theorem} \label{main} Let $R= S/I$ be as above. Then
\[
\ds \ind(R) = 5 < 6 = \gll(R).
\]
\end{theorem}
\begin{proof} Using either CoCoA or Macaulay2 we get
\[
\ds y\m^5: \m = (y^5,xy^3,yz^4,xyz^3,y^3z^2,xy^2z^2,y^4z)R \subseteq yR
\]
and 
\[
\ds y\m^4:\m = (y^4,yz^3,y^2z^2,xyz^2,y^3z,xy^2z,xz^4)R \not\subseteq yR.
\]
Therefore $\ind(R) = 5$ by Proposition \ref{colon}. On the other hand, it is easy to see that $\m^6 \subseteq yR$, therefore $\gll(R) \ls 6$, and we want to show that equality holds. To do this, we need to prove that for any $f \in \m \smallsetminus\{0\}$ we have $\ll(R/fR) \gs 6$, or equivalently that $\m^5 \not\subseteq fR$. Assume the contrary, i.e., that there exists $f \in \m$ such that $\m^5 \subseteq fR$. Lifting to $S$ we get that $\n^5 \subseteq I+(f) = :J$, that is
\[
\ds J = J + \n^5 = (x^2,xy^2+yz^3,f) + \n^5.
\]
First, assume that $\ord(f) = 1$, and let $f^*$ be the initial form of $f$. Since $\ord(f) =1$, $f$ can be made a part of a minimal set of generators of $\n$: say that $f,a,b \in \n$ are linearly independent modulo $\n^2$. Furthermore, we have
\[
\ds J \cap \n^2 + \n^3 = (x^2,f^2,af,bf) + \n^3,
\]
so that
\[
\ds 3 \ls \lambda((J^*)_2) = \lambda\left(\frac{J \cap \n^2 + \n^3}{\n^3}\right) \ls 4,
\]
and such length is equal to three if and only if $x^2 \in (f^2,af,bf) + \n^3$. But, by choice of $f,a,b$, this happens if and only if $f = u x+g$, for some unit $u \in S$ and some some $g \in S$ of order at least two. Then, for $h=u^{-1}g$, we have
\[
\ds \n^5 + (f) \subseteq J = (x^2,xy^2+yz^3,ux+g) + \n^5 = (h^2,-hy^2+yz^3,ux+g) + \n^5 \subseteq \n^4+(f).
\]
Notice that $\lambda\left(\frac{\n^4+(f)}{\n^5+(f)}\right) = 5$, and
\[
\ds \lambda\left(\frac{J}{\n^5+(f)}\right) = \lambda \left(\frac{(h^2,-hy^2+yz^3)+(\n^5+(f))}{\n^5+(f)}\right) \ls 2,
\]
therefore $\lambda\left(\frac{\n^4+(f)}{J}\right) \gs 3$. But this is a contradiction, because $\frac{\n^4+(f)}{J} \subseteq \soc\left(S/J\right) = \soc(R/fR)$, which is simple because $R$ is Gorenstein. We ruled out the case $\lambda((J^*)_2) =3$, so we are left with the case $\lambda((J^*)_2) = 4$. Under such condition, the Hilbert function of $R/fR$, which is the Hilbert function of $P/J^*$, is
\begin{eqnarray*}
\begin{tabular}{c|c|c|c|c|c|c|c|c|c}
\ $n$ \ & \ $0$ \ & \ $1$ \ & \ $2$ \ & \ $3$ \ & \ $4$ \ & \ $5$ \ & \ $6$ \ & \ $7$ \ & $\ldots$ \\
\hline
$HF_{R/fR}(n)$ & $1$ & $2$ & $2$ & $h$ & $k$ & $0$ & $0$ & $0$ & $\ldots$
\end{tabular}
\end{eqnarray*}
where $k\ls 1$ because $R$ is Gorenstein and because $\m^5 \subseteq fR$ by assumption, and $h\ls 2$ by Macaulay's Theorem \ref{Mac}. On the other hand, $\lambda(R/fR) \gs e(R) = 8$ by Lemma \ref{sup}, therefore we necessarily have $k=1$ and $h = HF_{R/fR}(3) = HF_{P/J^*}(3) =2$. In addition, again by Lemma \ref{sup}, $f$ must be a superficial element. Since
\[
I^* = (X^2,Y) \cap (X,Z^6) \cap (X^2,Y^2,Z^3)
\]
and $f$ is superficial if and only if $f^* \notin \bigcup_{\p \in \Min(G)} \p = (X,Y)G \cup (X,Z)G$ by Remark \ref{par}, we conclude that $f^* \notin (X,Y)$. Let $\MM = (X,Y,Z)$ be the irrelevant maximal ideal of $P$. Then, since $f^* \in \MM \smallsetminus (\MM^2 \cup (X,Y))$, we have that $(f^*,X,Y) = \MM$. Let $K:=I^*+(f^*) \subseteq P$, then we have
\[
\ds \lambda\left(K_3 \right) \gs \lambda \left(\frac{(X^3,X^2Y,XY^2) + f^*\MM^2}{\MM^4} \right) = 9.
\]
But $HF_P(3) = 10$, therefore $HF_{P/K}(3) \ls 1$. In particular, we see that $HF_{P/K}(3) < HF_{P/J^*}(3)$. On the other hand, there is always a surjective homomorphism of graded rings 
\begin{equation}
\label{surj}
\ds P/K \to P/J^* \to 0,
\end{equation}
which is homogeneous of degree zero (see for example \cite[Corollary 8.6.2]{HunSw}). This gives the desired contradiction. We analyzed all possible cases when $\ord(f) = 1$, so let us assume now that $\ord(f) = 2$. Again, let $J:=I+(f)$ and $K:=I^*+(f^*)$. In this case $J \cap \n^2 + \n^3 = (x^2,f) + \n^3$, so that
\[
\ds 1 \ls \lambda((J^*)_2) = \lambda\left(\frac{(x^2,f) + \n^3}{\n^3}\right) \ls 2,
\]
and such length is equal to one if and only if  $f = ux^2+g$ for some unit $u \in S$ and $g$ of order at least three. Assume that we are in the latter case, then the Hilbert function of $R/fR$, which is the Hilbert function of $P/J^*$, is
\begin{eqnarray*}
\begin{tabular}{c|c|c|c|c|c|c|c|c|c}
\ $n$ \ & \ $0$ \ & \ $1$ \ & \ $2$ \ & \ $3$ \ & \ $4$ \ & \ $5$ \ & \ $6$ \ & \ $7$ \ & $\ldots$ \\
\hline
$HF_{R/fR}(n)$ & $1$ & $3$ & $5$ & $h$ & $ k$ & $0$ & $0$ & $0$ & $\ldots$
\end{tabular}
\end{eqnarray*}
with $k  = HF_{R/fR}(4) \ls 1$, because $R/fR$ is a zero-dimensional Gorenstein local ring and $\m^5 \subseteq fR$. By Macaulay's Theorem \ref{Mac} we must have $h \ls 7$. Also $\lambda(R/fR) > 16$ by Lemma \ref{sup}, because $(f^*) = (X^2)$ is contained in the minimal prime $(X,Y)$ of $G$, so that $f$ cannot be superficial. Therefore $h=HF_{R/fR}(3) = HF_{P/J^*}(3) = 7$ and $k=1$ are forced. On the other hand, $(f^*) = (X^2) \subseteq I^*$, therefore $K=I^*$ and $HF_{P/K}(3) = HF_G(3) = HF_R(3) = 6$. This is again a contradiction because of the surjection (\ref{surj}). Thus we can assume that $\lambda((J^*)_2) = 2$. In this case the Hilbert function of $R/fR$ is
\begin{eqnarray*}
\begin{tabular}{c|c|c|c|c|c|c|c|c|c}
\ $n$ \ & \ $0$ \ & \ $1$ \ & \ $2$ \ & \ $3$ \ & \ $4$ \ & \ $5$ \ & \ $6$ \ & \ $7$ \ & $\ldots$ \\
\hline
$HF_{R/fR}(n)$ & $1$ & $3$ & $4$ & $h$ & $k$ & $0$ & $0$ & $0$ & $\ldots$
\end{tabular}
\end{eqnarray*}
with $k \ls 1$. In addition, Macaulay's Theorem \ref{Mac} implies that $h\ls 5$. Thus $\lambda(R/fR) \ls 14 < 16 = 2e(R)$, contradicting Lemma \ref{sup}.  Finally, let us assume that $\ord(f) \gs 3$, so that $\lambda(R/fR) \gs 24$ by Lemma \ref{sup}. Since $HF_{R/fR}(4) \ls 1$ because $R$ is Gorenstein and we are assuming that $\m^5 \subseteq fR$, the maximal possible Hilbert function for $R/fR$ is
\begin{eqnarray*}
\begin{tabular}{c|c|c|c|c|c|c|c|c|c}
\ $n$ \ & \ $0$ \ & \ $1$ \ & \ $2$ \ & \ $3$ \ & \ $4$ \ & \ $5$ \ & \ $6$ \ & \ $7$ \ & $\ldots$ \\
\hline
$HF_{max}(n)$ & $1$ & $3$ & $6$ & $10$ & $1$ & $0$ & $0$ & $0$ & $\ldots$
\end{tabular}
\end{eqnarray*}
But then $\lambda(R/fR) \ls 21 < 24$, a contradiction. This shows that for any $f \in \m \smallsetminus\{0\}$ we have $\m^5 \not\subseteq fR$, and completes the proof that $\gll(R) = 6 > 5 = \ind(R)$.
\end{proof}
\section{Further examples and remarks}
Given that the conjecture fails in general, even for complete intersection domains of dimension one, one may wonder if the conjecture is true for some smaller classes of rings. A ring $R$ is called quasi-homogeneous if it isomorphic to the completion of a positively graded $k$-algebra at the irrelevant maximal ideal. In the counterexample of Section 3, it is not possible to give weights to the variables $x,y$ and $z$ that make the completion $\widehat R$ of $R$ quasi-homogeneous. This does not say that $R$ is not quasi-homogeneous, since such weights could exist for a different choice of minimal generators of the maximal ideal $\m$, however we were not able to find it. On the other hand, it is easy to find a quasi-homogeneous counterexample if one is not looking for a domain:
\begin{example} \label{ex1} Let $S=k\ps{x,y,z}$, where $k$ is any field, and let $\n = (x,y,z)$ be the maximal ideal of $S$. Consider the one dimensional complete intersection $R:=S/I$, where
\[
\ds I = (x^2-y^5,xy^2+yz^3)S.
\]
The ring $R$ is quasi-homogeneous, since it is the completion of the positively graded ring $k[x,y,z]/(x^2-y^5,xy^2+yz^3)$ with weights $w(x) = 15$, $w(y) = 6$ and $w(z) = 7$ at the maximal ideal $(x,y,z)$. Using CoCoA or Macaulay2, one sees that 
\[
\ds z\m^5:\m = (xy^2z,z^5,xz^4,y^3z^2,y^4z,y^2z^3,xyz^3)R \subseteq zR,
\]
and
\[
\ds z\m^4:\m = (xy^2,x^2y,z^4,xz^3,y^2z^2,xyz^2,y^3z)R \not\subseteq zR.
\]
Therefore $\ind(R) = 5$, and because $\m^6 \subseteq (y-z)$ we also have that $\gll(R)\ls 6$. On the other hand, note that $I+\n^5 = (x^2,xy^2+yz^3)+\n^5$, $I^*= (X^2,XY^2,XYZ^3,YZ^6) \subseteq k[X,Y,Z]$ and $e(R) = 8$ are the same as in the counterexample of Section 3. Therefore, using the same proof of Theorem \ref{main}, we get $\gll(R) = 6$. 
\end{example}
\begin{remark} \label{remDing} In \cite[Corollary 3.3]{Ding1}, Ding claims that the conjecture is true for what he calls gradable rings such that $\Depth({\rm gr}_\m(R)) \gs \dim(R)-1$. In our notation, gradable rings correspond to quasi-homogeneous rings. His argument is not correct, since he uses results that need a standard grading, i.e. the weights of all the minimal generators of $\m$ must be one. Example \ref{ex1} gives a counterexample to his statement. 
\end{remark}
Another direction of investigation is to consider Gorenstein analytically irreducible rings, i.e. local rings such that the completion at the maximal ideal is a domain. We thank William Heinzer for pointing out that the counterexample in Section 3 is not such:
\[
\ds \widehat R \cong \frac{k\ps{x,y,z}}{(x^2-y^5,xy^2+yz^3-z^5)} \cong \frac{k\ps{t^2,t^5,z}}{(z^5-t^2z^3 - t^9)} \subseteq \frac{k\ps{t,z}}{(z^5-t^2z^3 - t^9)}:= T,
\]
is an integral extension. The inclusion follows from the fact that $t^4$ is in the conductor $\widehat{R}:_{\widehat{R}} T$, and it is a non zero-divisor in $\widehat{R}$. The initial form $z^5-t^2z^3$ has two relatively prime non-constant factors $z^3$ and $z^2-t^2$ in $\QQ\ps{t,z}$, therefore $T$ is not a domain \cite[Theorem 16.6]{KunzAlgCurves}. Since $\widehat{R}$ and $T$ have the same total ring of fractions, $\widehat R$ is also not a domain. 

However, there is an analytically irreducible counterexample:
\begin{example} \label{ex2} Let $S=\QQ[x,y,z]_{(x,y,z)}$ and let $\n = (x,y,z)S$ be the maximal ideal of $S$. Consider the one dimensional domain
\[
\ds R:=\QQ[t^8+t^{10},t^9,t^{20}+t^{36}]_{(t^8+t^{10},t^9,t^{20}+t^{36})}
\]
and let $\m=(t^8+t^{10},t^9,t^{20}+t^{36})R$ be its maximal ideal. Using CoCoA or Macaulay2 one sees that $R \cong S/I$, where
\[
\ds I = (z^2+f_1,y^4-x^2z+2y^2z+z^2+f_2) \subseteq S
\]
for some $f_1,f_2 \in \n^5$. In particular, $R$ is a complete intersection. We checked with Macaulay2 and CoCoA that $\m^6 \subseteq xR$, $x\m^5:\m \subseteq xR$ and $x\m^4:\m \not\subseteq xR$, hence $\ind(R)= 5$ and $\gll(R) \ls 6$. On the other hand, there does not exist $f \in S$ such that $\n^5 \subseteq I + (f)$, otherwise $I+(f) = (z^2,y^4-x^2z+2y^2z,f) + \n^5$ and since $e(R) = 8$ one can use arguments which are analogous to the ones used in the proof of Theorem \ref{main} to show that this cannot happen. Therefore $\gll(R) = 6$. Finally, to show that $R$ is analytically irreducible, notice that $R \subseteq \QQ[t]_\m:= V$ is an integral birational extension, and since $V$ is normal we have in fact that $V$ is the integral closure of $R$ in its field of fractions. The ring $V$ is semi-local, with maximal ideals $N_1,\ldots,N_s$. Since $N_i \cap R = \m$, each $N_i$ contains $t^9$, and hence it must contain $t$. So $V$ has only one maximal ideal, namely $(t)V$, and thus it is local. There is a well known one to one correspondence between maximal ideals in the integral closure of $R$ and minimal primes in $\widehat R$, therefore $\widehat R$ is a domain. 
\end{example}
%$\ds t= \frac{[t^{20}+t^{36}] +[t^9]^2- [t^9]^4}{[t^9][t^8+t^{10}]}$
\begin{remark} If $R$ is an equicharacteristic one dimensional complete local domain, the integral closure $\ov R$ of $R$ in its quotient field is isomorphic to a power series ring $k\ps{t}$. Thus, $R$ is of the form $k\ps{f_1,\ldots,f_n}$ for some $f_1,\ldots,f_n \in k\ps t$. Ding proved that the conjecture is true when the $f_i$'s are monomials in $t$ \cite[Proposition 2.6]{Ding}. This is no longer true if the $f_i$'s are not monomials, as Example \ref{ex2} shows. In fact, with the notation introduced above, we have that $\widehat R$ is a one dimensional complete equicharacteristic local domain, and by \cite[Lemma 3.3 and Corollary 5.2]{HaSh} we have that
\[
\ds \ind(\widehat R) = \ind(R) = 5 < 6 = \gll(R) = \gll(\widehat R).
\]
\end{remark}
Given the results in this article, it seems natural to ask the following questions.

\begin{questions} \label{quests} Let $(R,\m,k)$ be a Gorenstein local ring with infinite residue field.
\begin{enumerate}[(i)]
\item Are $\ind(R) = 5$ and $\gll(R) = 6$ minimal possible values for an example where $\ind(R) \ne \gll(R)$?
\item Is $\gll(R)$ always attained by a system of parameters that generates a reduction of $\m$?
\end{enumerate}
\end{questions}
\begin{remark} We conclude by making some comments about the role of computer algebra programs in the proofs contained in this article. In Theorem \ref{main}, we justify certain arguments, for instance the inclusion $y\m^5:\m \subseteq yR$, the initial ideal $I^*$ and its primary decomposition, and the Hilbert function of $R$, by saying that we checked them with CoCoA and Macaulay2. However, we also verified the validity of these statements by hand. These are just tedious computations, adding no real content to the argument, and we decided not include them in this article. Analogous considerations apply to Example \ref{ex1}. On the other hand, in Example \ref{ex2}, both the equations in $S$ defining $R$ and the inclusion $x\m^5:\m \subseteq xR$ are rather hard computationally, hence the correctness of Example \ref{ex2} relies heavily on a careful analysis of the answers given by CoCoA and Macaulay2.
\end{remark}

\section*{Acknowledgements}
The author is extremely grateful to Craig Huneke for bringing the problem to his attention, for sharing insightful ideas and for constant support. The author would also like to thank the referee for very useful comments, and for suggesting Questions \ref{quests}.

\bibliographystyle{alpha}
\bibliography{References}
{\footnotesize

\vspace{0.5in}

\noindent \small \textsc{Department of Mathematics, University of Virginia, Charlottesville, VA  22903} \\ \indent \emph{Email address}:  {\tt ad9fa@virginia.edu} 
}

\end{document}